\documentclass[amsmath,10pt]{article}

\usepackage{graphics}
\usepackage[usenames]{color}
\usepackage{amsmath}
\usepackage{amsthm}
\usepackage{amsbsy}

\newtheorem{theo}{Theorem\bf}

\newtheorem{definition}{Definition\bf}

\begin{document}

\bibliographystyle{plain}

\large \begin{center}A GENERALIZED CLOSED FORM FOR TRIANGULAR MATRIX POWERS \normalsize

Walter Shur
\end{center}

\begin{center}20 Speyside Circle, Pittsboro, North Carolina, 27312, USA

Tel: 919-542-7179 \ \ Email: wrshur@gmail.com

Affiliation: None (Retired) \end{center}\vspace{.2in}

\begin{center}\bf ABSTRACT.\end{center}

 \rm  Given a triangular matrix $M=[m_{i,j}]$,  this paper shows how to obtain numbers $c_{i,j,r,s}$ such that the $(i,j)^{th}$ element of $M^n$ is given by $_nm_{i,j}=\displaystyle\sum_{r=1}^{num(i,j)}\sum_{s=1}^{mpy_{i,j}(r)}c_{i,j,r,s} \binom{n-1}{s-1}$ $m_{i,j,r}^{n-s}$, where $
num(i,j)$ is the number of unique diagonal elements between and including the $i^{th}$ and $j^{th}$ rows, $\{ m_{i,j,r}\}_1^{num(i,j)}$ is the set of those unique elements, and $mpy_{i,j}(r)$ is the multiplicity of $m_{i,j,r}$ on that same range.  The $c_{i,j,r,s}$ are independent of the power to which the matrix is raised. This generalized formula works for any power of $M$, negative, zero or positive (positive only, if the matrix is singular).\rm

\vspace{.3in}



\vspace{.3in}

\bf 1. INTRODUCTION\rm 

 [1] presents a method of obtaining a simple closed form for the powers of a triangular matrix with unique diagonal elements, as follows:
\vspace{.1in}

\begin{definition}
Let $M=[m_{i,j}]$ be a $k\times k$ upper triangular matrix with unique diagonal elements. We define the power factors of $M$, $p_{i,j,s}$, recursively on the index $j$\ ,as follows:

\vspace{.3 in}

$p_{i,j,s}=\frac {\displaystyle \sum_{t=s}^{j-1}p_{i,t,s}m_{t,j}}{m_{s,s}-m_{j,j}}$\hspace{1 in} $i\leq s <j\leq k$,\hspace{.5 in}(1.1)

\vspace{.3 in}
$p_{i,j,s}=0 \hspace{1.85in} s<i,\ s>j$,

\vspace{.3 in}
\nopagebreak
$p_{i,j,j}=m_{i,j}-\displaystyle\sum_{t=i}^{j-1}{p_{i,j,t}} \hspace{.95in} i<j\le k$,\hspace{.6 in}(1.2)

\vspace{.3 in}

$p_{j,j,j}=m_{j,j}$.

\end{definition}

\begin{theo}

If  $M=[m_{i,j}]$ is a non-singular upper triangular matrix with unique diagonal elements, and $_{n}m_{i,j}$is the $(i,j)^{th}$ element of $M^n$, then  
\[_{n}m_{i,j}=\displaystyle\sum_{s=i}^{j}p_{i,j,s}m_{s,s}^{n-1},\]
for all integral values of $n$, negative, positive or zero. If $M$ is singular, the equation holds if $n\geq 1$ and $0^0$ is taken as $1$. 
\end{theo}

\vspace{.3 in}
\bf 2. ALTERNATE DEFINITION FOR $p_{i,j,s}$.\rm

 Let $M=[m_{i,j}]$ be an upper triangular matrix with unique diagonal elements. The product \[ m_{i,a}m_{a,b}m_{b,c}\cdots m_{l,j},\] where $i\le a <b<c \cdots<l<j$,  and $s$ is $j$ or  one of $a$,$b$,$c$,$\cdots$ $l$, is  called a chain from $i$ to $j$  passing through $s$. The length of the chain is the number of elements in the product. The expression \[\frac{ m_{i,a}m_{a,b}m_{b,c}\cdots m_{l,j}}{(m_{s,s}-m_{a,a})(m_{s,s}-m_{b,b})(m_{s,s}-m_{c,c})\cdots (m_{s,s}-m_{l,l})(m_{s,s}-m_{j,j})}\]
where $(m_{s,s}-m_{s,s})$ is taken as $1$, is called  an adjusted chain from $i$ to $j$  passing through $s$.

\begin{definition}
 If  $i\le s \le j$, $ p_{i,j,s}$ is the sum of all adjusted chains from $i$ to $j$  passing through $s$.
If $s<i$ or $s>j$, $ p_{i,j,s}=0$. 
\vspace{.2 in}

\rm Following are a few illustrative examples which help clarify the definition:
\vspace{.1 in}

$p_{1,1,1}=m_{1,1},$ 
\vspace{.2 in}

$p_{1,3,1}=\frac{m_{1,1}m_{1,3}}{(m_{1,1}-m_{3,3})}+\frac{m_{1,1}m_{1,2}m_{2,3}}{(m_{1,1}-m_{2,2})(m_{1,1}-m_{3,3})}$,
\vspace{.2 in}

$p_{1,3,2}=\frac{m_{1,2}m_{2,3}}{m_{2,2}-m_{3,3}}+\frac{m_{1,1}m_{1,2}m_{2,3}}{(m_{2,2}-m_{1,1})(m_{2,2}-m_{3,3})}$,
\vspace{.2 in}

$p_{1,3,3}=m_{1,3}+\frac{m_{1,1}m_{1,3}}{m_{3,3}-m_{1,1}}+\frac{m_{1,2}m_{2,3}}{m_{3,3}-m_{2,2}}+\frac{m_{1,1}m_{1,2}m_{2,3}}{(m_{3,3}-m_{1,1})(m_{3,3}-m_{2,2})}$.

\end{definition} 
\begin{theo}
Definition 2 is equivalent to Definition 1.
\end{theo}

\begin{proof}
From Definition 2, each term of the  summand in (1.1) is of the form \[\frac{ m_{i,a}m_{a,b}m_{b,c}\cdots m_{l,t}m_{t,j}}{(m_{s,s}-m_{a,a})(m_{s,s}-m_{b,b})(m_{s,s}-m_{c,c})\cdots (m_{s,s}-m_{l,l})(m_{s,s}-m_{t,t})},\] where $(m_{s,s}-m_{s,s})$ is taken as $1$. The sum of  all such terms, from $t=s$ to $t=j-1$,  clearly includes all of the chains from $i$ to  $j$ passing through $s$. That sum  would be the sum of all  the adjusted chains from $i$ to $j$ passing through $s$, except that the difference $(m_{s,s}-m_{j,j)}$ is missing from each denominator. Hence the division by that difference in (1.1), and thus (1.1) is satisfied.

Next, we need to show that (1.2) is satisfied.  The sum $\displaystyle\sum_{t=i}^{j}p_{i,j,t}$ consists of all adjusted chains from $i$ to $j$ of the form \[\frac{ m_{i,a_1}m_{a_1,a_2}m_{a_2,a_3}\cdots m_{a_r,j}}{(m_{s,s}-m_{a_1,a_1})(m_{s,s}-m_{a_2,a_2})(m_{s,s}-m_{a_3,a_3})\cdots (m_{s,s}-m_{a_r,a_r})(m_{s,s}-m_{j,j})},\]with values  $r=1,2,3,\cdots,j-i$ and  $s=a_1,a_2,a_3,\cdots,j$ . The only term with \mbox{$r=1$} comes from $p_{i,j,j}$ and is equal to $m_{i,j}$ (see the illustrative example for $p_{1,3,3}$ in  Definition 2). The sum of all the adjusted chains of length $2$ with the same numerator is
\vspace{.2 in}

$  {m_{i,a_1}m_{a_1,a_2}(\frac{1}{m_{a_1,a_1}-m_{a_2,a_2}}+\frac{1}{m_{a_2,a_2}-m_{a_1,a_1}})=0.}$ 
\vspace{.2 in}

The sum of all adjusted chains of length $3$ with the same numerator is
\vspace{.2 in}

 $ m_{i,a_1}m_{a_1,a_2}m_{a_2,a_3}(\frac{1}{(m_{a_1,a_1}-m_{a_2,a_2})(m_{a_1,a_1}-m_{a_3,a_3})}+\frac{1}{(m_{a_2,a_2}-m_{a_1,a_1})(m_{a_2,a_2}-m_{a_3,a_3})}$+

$\frac{1}{(m_{a_3,a_3}-m_{a_1,a_1})(m_{a_3,a_3}-m_{a_2,a_2})})=0.$
\vspace{.2 in}

In general, the multiplier of $m_{i,a_1}m_{a_1,a_2}m_{a_2,a_3}\cdots m_{a_r,j}$ is seen \cite{J} to be the $(r-1)^{st}$ divided difference of the polynomial $f(x)=1$, and hence is $0$ if $r\ge 2 $.

 Therefore,$\displaystyle \sum_{t=i}^{j}p_{i,j,t}=m_{i,j}$ , and (1.2) is satisfied. And since \mbox{$p_{j,j,j}=m_{i,j}$}, Definition 2 is equivalent to Definition 1.
 
\vspace{.3 in}

\end{proof}

\bf 3. NON-UNIQUE DIAGONAL ELEMENTS\rm

\begin{theo}

Let $M=[m_{i,j}]$ be an upper $k\times k$ triangular matrix  with non-unique diagonal elements. Let  $
num(i,j)$ be the number of unique diagonal elements between and including the $i^{th}$ and $j^{th}$ rows, $\{ m_{i,j,r}\}_1^{num(i,j)}$ be the set of those unique elements, and $mpy_{i,j}(r)$  be the multiplicity of $m_{i,j,r}$ on that same range ($\displaystyle \sum_{t=1}^{num(i,j)}mpy_{i,j}(r)=k$). 

Then \[_nm_{i,j}=\sum_{r=1}^{num(i,j)}\sum_{s=1}^{mpy_{i,j}(r)} c_{i,j,r,s}\binom{n-1}{s-1}\,\, {
m_{i,j,r}^{n-s}},\]
where the $c_{i,j,r,s}$ are independent of the power to which the matrix is raised. In particular, 
\[c_{i,j,r,s}=\left [p_{i,j,i_1}x^{i_1(s-1)}+p_{i,j,i_2}x^{i_2(s-1)}+ \cdots +p_{i,j,i_{mpy_{i,j}(r)}}
x^{i_{mpy_{i,j}(r)}(s-1)} \right ]_{x=0}, \]
where ${(i_1,i_1)}, ({i_2,i_2)}, \cdots, {(i_{mpy_{i,j}(r)},i_{mpy_{i,j}(r)})}$
are the   $mpy_{i,j}(r)$ locations where the value $m_{i,j,r}$ appears, and the $p_{i,j,i_1}, p_{i,j,i_2}\cdots$ are determined from the matrix $M$ with $m_{i,i}$ replaced by $m_{i,i}+x^i$.

This generalized formula works for all integral values of $n$, negative, positive or zero. If $M$ is singular, the formula holds if $n\geq 1$ and $0^0$ is taken as $1$.
\end{theo}
\begin{proof}

We first alter the matrix $M$ by replacing each diagonal element $m_{i,i}$ by $m_{i,i}+x_i$, where the $x_i$ are variables which will be changed to $x^i$ and set to zero at the final stage to obtain results for the unaltered matrix $M$. The purpose of this alteration is to prevent any denominators from becoming zero during the development of final formulas.

We have from Theorem 1,

\[_{n}m_{i,j}=\displaystyle\sum_{s=i}^{j}p_{i,j,s}(m_{s,s}+x_s)^{n-1},\hspace{.9in}(3.1)\]

To make the notation in what follows easier to read, let \mbox{$m_{i,j,r}=m$} and $mpy_{i,j}(r)=t$, that is, the value $m$ appears exactly $t$ times among the diagonal elements between and including the $i^{th}$ and $j^{th}$ rows. Let ${(i_1,i_1)}, ({i_2,i_2)}, \cdots, {(i_t,i_t)}$ be the $t$ locations where the value $m$ appears.

Consider only those terms of the sum in (3.1) where $m_{s,s}=m_{i,j,r}=m$. We obtain

$p_{i,j,i_1}(m+x_{i_1})^{n-1}+p_{i,j,i_2}(m+x_{i_2})^{n-1}+\cdots +p_{i,j,i_t}(m+x_{i_t})^{n-1}$

Expanding the binomials, and collecting powers of $m$, we have

$\displaystyle \sum_{s=1}^n \left[p_{i,j,i_1}x_{i_1}^{s-1}+p_{i,j,i_2}x_{i_2}^{s-1}+\cdots+p_{i,j,i_t}x_{i_t}^{s-1}\right]\binom{n-1}{s-1}m^{n-s}$\hspace{.9in}(3.2)
\vspace{.2in}

Recall from Definition 2 that $p_{i,j,s}$ is the sum of all adjusted chains from $i$ to $j$ passing through $s$. If $t>1$, each of the adjusted chains in $p_{i,j,i_1}, p_{i,j,i_2}\cdots$ will contain exactly $t-1$ factors in the denominator which become zero when $x_{i_1}=x_{i_2}\cdots =x_{i_t}=0$. For example, the adjusted chains making up $p_{i,j,i_1}$ will contain the factors \mbox{ $(x_{i_1}-x_{i_2})(x_{i_1}-x_{i_3})\cdots (x_{i_1}-x_{i_t})$} in the denominator. We need to show that, nevertheless, the bracketed sum in (3.2) is not indeterminate when $x_{i_1}=x_{i_2}\cdots =x_{i_t}=0$.

We begin by replacing each $p_{i,j,s}$ in (3.2) by its representation as a sum of adjusted chains. We then separate all the bracketed terms into mutually exclusive sets, each set consisting of of all the terms where the numerator of the adjusted chain passes through a particular set of points. For example, one such set would consist of all the terms where the numerator of the adjusted chain passes through the points 1,3,4 and 9. It is not difficult to show that, for each such set, the sum of all its terms is not indeterminate. We illustrate the method for a set where  where the numerator, $N$, of the adjusted chain passes through the points $a,b,c,d,e,f$, where $m_{a,a}=m_{b,b}=m_{c,c} =m_{d,d}=m$, $m_{e,e}\ne m$, and $m_{f,f}\ne m$.

Noting, for example, that $(m_{a,a}+x_a)-(m_{b,b}+x_b)=(x_a-x_b)$, and letting $f(x_s)=((m_{s,s}+x_s)-(m_{e,e}+x_e))((m_{s,s}+x_s)-(m_{f,f}+x_f))$, the sum of the terms in that set is

\begin{eqnarray*}N \left [x_a^{s-1}/(x_a-x_b)(x_a-x_c)(x_a-x_d)f(x_a) + x_b^{s-1}/(x_b-x_a)(x_b-x_c)(x_b-x_d)f(x_b)\right.\\
 \left. +x_c^{s-1}/(x_c-x_a)(x_c-x_b)(x_c-x_d)f(x_c)+x_d^{s-1}/(x_d-x_a)(x_d-x_b)(x_d-x_c)f(x_d)\right ] \end{eqnarray*}

Combining the bracketed terms into a single fraction with the denominator $(x_a-x_b)(x_a-x_c)(x_a-x_d)(x_b-x_c)(x_b-x_d)(x_c-x_d)f(x_a)f(x_b)f(x_c)f(x_d)$, the numerator is

\begin{eqnarray*} (x_b-x_c)(x_b-x_d)(x_c-x_d)f(x_b)f(x_c)f(x_d)x_a^{s-1}\\
-(x_a-x_c)(x_a-x_d)(x_c-x_d)f(x_a)f(x_c)f(x_d)x_b^{s-1}\\
(x_a-x_b)(x_a-x_d)(x_b-x_d)f(x_a)f(x_b)f(x_d)x_c^{s-1}\\
-(x_a-x_b)(x_a-x_c)(x_b-x_c)f(x_a)f(x_b)f(x_c)x_d^{s-1}
\end{eqnarray*}

Note that if $x_a=x_b$, the numerator is zero. Hence, $x_a-x_b$ is a factor of the numerator and cancels out the corresponding factor in the denominator. The same is true for each other pair of $x_a, x_b, x_c, x_d$.  And since none of $f(a),f(b),f(c),f(d)$ is zero when $x_{i_1}=x_{i_2}\cdots =x_{i_t}=0$, the bracketed sum is not indeterminate.

We show now in the following paragraphs that the bracketed sum in (3.2) is zero if $s>t$.

Again, replace each $p_{i,j,s}$ in (3.2) by its representation as a sum of adjusted chains, but this time combine all the bracketed terms into a single fraction, $F$, with the denominator $DZ$, where

 $D=\prod \limits_{\stackrel{p=1}{p\ne 1}}^t (x_{i_1}-x_{i_p})\prod \limits_{\stackrel{p=1}{p\ne 2}}^t (x_{i_2}-x_{i_p})\cdots \prod \limits_{\stackrel{p=1}{p\ne t}}^t (x_{i_t}-x_{i_p})$,

\noindent and $Z$ is the product of all factors in the denominators of the adjusted chains which do not become zero when $x_{i_1}=x_{i_2}\cdots =x_{i_t}=0$.
 
If the sum of the bracketed terms in (3.2) were equal to a constant different from zero when $x_{i_1}=x_{i_2}\cdots =x_{i_t}=0$, the numerator of $F$ would have to be equal to $A*D$, where $A$ is a constant not equal to zero. But note that, when $D$ is expanded, it contains the term $x_{i_1}^{t-1}x_{i_2}^{t-1}\cdots x_{i_t}^{t-1}$. This term cannot be in the numerator of $F$ when $s>t$, because every term in the numerator of $F$ would contain a power greater than $t-1$ of at least one of $x_{i_1},x_{i_2},\cdots x_{i_t}$. Therefore, when $s>t$, the numerator of $F$ cannot be $A*D, A\ne 0$; hence the sum of the bracketed terms in (3.2) is zero when $s>t$. 

Replacing in (3.2) $n$ by $t$ (and recalling that $t=mpy_{i,j}(r)$),  $m$ by  \mbox{$m_{i,j,r}$}, and the $x_i$ by $x^i$ , we see that the sum of the terms in (3.1) where $m_{s,s}=m_{i,j,r}$ is given by 

\hspace{-1.4in}$\displaystyle \sum_{s=1}^{mpy_{i,j}(r)} \left [p_{i,j,i_1}x^{i_1(s-1)}+p_{i,j,i_2}x^{i_2(s-1)}+ \cdots +p_{i,j,i_{mpy_{i,j}(r)}}x^{i_{mpy_{i,j}(r)}(s-1)} \right ]\binom{n-1}{s-1}m_{i,j,r}^{n-s}$, \hspace{.7in}(3.3)

\noindent where ${(i_1,i_1)}, ({i_2,i_2)}, \cdots, {(i_{mpy_{i,j}(r)},i_{mpy_{i,j}(r)})}$ are the $mpy_{i,j}(r)$ locations where the value $m_{i,j,r}$ appears. 

Summing (3.3) from $r$=1 to $num(i,j)$, and letting 

\[c_{i,j,r,s}=\left [p_{i,j,i_1}x^{i_1(s-1)}+p_{i,j,i_2}x^{i_2(s-1)}+ \cdots +p_{i,j,i_{mpy_{i,j}(r)}}
x^{i_{mpy_{i,j}(r)}(s-1)} \right ]_{x=0}, \] gives us Theorem 3.

\end{proof}

\bf 4. Illustrations\rm

In this section we give three examples of how to use Mathematica to obtain closed form formulas. The first step, of course, in using Mathematica is to define the altered matrix and (using Definition 1) the $p$ factors.

(1) let $M=[m_{i,j}]$ be the matrix
\vspace{.2in}

 \[ M= \left(\begin{array}{cccccc}
3&2&3&5&4&2\\
0&5&2&4&3&1\\
0&0&3&2&6&4\\
0&0&0&5&5&1\\
0&0&0&0&7&2\\
0&0&0&0&0&3
\end{array}\right)\]
In terms of Theorem 3,
\begin{eqnarray*}
m_{1,6,1}=3&mpy_{1,6}(1)=3,\\
m_{1,6,2}=5&mpy_{1,6}(2)=2,\\
m_{1,6,3}=7&mpy_{1,6}(3)=1.
\end{eqnarray*}
\nopagebreak and  $num(1,6)=3$.
In order to make use of Theorem 3, we first alter the matrix $M$ so that it will have unique diagonal elements, as follows:
\vspace{.3in}

$ M= \left(\begin{array}{cccccc}
3+x&2&3&5&4&2\\
0&5+x^2&2&4&3&1\\
0&0&3+x^3&2&6&4\\
0&0&0&5+x^4&5&1\\
0&0&0&0&7+x^5&2\\
0&0&0&0&0&3+x^6
\end{array}\right)$

\vspace{.3in}

The following paragraphs show how to  find a closed form formula for $_nm_{1,6}$. 
\vspace{.2in}

Determine the coefficients of powers of $3$:
\vspace{.2in}

\hspace{-.2in}\mbox{$c_{1,6,1,1}$ 
Input: $(p[1,6,1]+p[1,6,3]+p[1,6,6]$//Together)/.$\{x->0\}$}

\hspace{.2in} Output: $-\frac{203}{32}$
\vspace{.2in}

\hspace{-.2in}\mbox{$c_{1,6,1,2}$ Input: $(p[1,6,1]x+p[1,6,3]x^\wedge3+p[1,6,6]x^\wedge6$//Together)/.$\{x->0\}$}

\hspace{.2in} Output: $\frac{5}{8}$
\vspace{.2in}

\hspace{-.2in}\mbox{$c_{1,6,1,3}$ Input: $(p[1,6,1]x^\wedge2+p[1,6,3]x^\wedge6+p[1,6,6]x^\wedge12$//Together)/.$\{x->0\}$}

\hspace{.2in} Output: $\frac{15}{2}$
\vspace{.3in}

Determine the coefficients of powers of $5$:

\vspace{.2in}

$c_{1,6,2,1}$ Input: $(p[1,6,2]+p[1,6,4]$//Together)/. $\{x->0\}$

\hspace{.4in} Output: $-\frac{59}{2}$

\vspace{.1in}

$c_{1,6,2,2}$ Input:$((p[1,6,2]x^\wedge2 +p[1,6,4]x^\wedge4)$//Together)/. $\{x->0\}$

\hspace{.4in} Output: $-60$

\vspace{.2in}

Determine the coefficient of powers of $7$: 
\vspace{.2in}

$c_{1,6,3,1}$ Input: (p[1,6,5]//Together)/.$\{x->0\}$

\hspace{.4in} Output: $\frac{1211}{32}$

\vspace{.2in}

Combining  the above results, we have from Theorem 3,

\begin{eqnarray*}_nm_{1,6}=-\frac{203}{32}3^{n-1}+\frac{5}{8}\binom{n-1}{1}3^{n-2}+\frac{15}{2}\binom{n-1}{2}3^{n-3}-\\[.1in]
\frac{59}{2}5^{n-1}-60\binom{n-1}{1}5^{n-2}+\frac{1211}{32}7^{n-1}.
\end{eqnarray*}

In the terms of Theorem 3, this is
\begin{eqnarray*}_nm_{1,6}=c_{1,6,1,1}m_{1,6,1}^{n-1}+c_{1,6,1,2}\binom{n-1}{1}m_{1,6,1}^{n-2}+c_{1,6,1,3}\binom{n-1}{2}m_{1,6,1}^{n-3}+\\[.1in]
c_{1,6,2,1}m_{1,6,2}^{n-1}+c_{1,6,2,2}\binom{n-1}{1}m_{1,6,2}^{n-2}+c_{1,6,3,1}m_{1,6,3}^{n-1}.
\end{eqnarray*}

\vspace{.2in}

(2) The following paragraphs show how to  find a closed form formula for $_nm_{2,4}$ for the  matrix $M$ above, noting that $m_{2,4,1}=5$, $mpy_{2,4}(1)=2$, $m_{2,4,2}=3$, $mpy_{2,4}(2)=1$, and $num(2,4)=2$. 
\vspace{.2in}

Determine the coefficient of powers of $5$:
\vspace{.2in}

$c_{2,4,1,1}$ Input: $(p[2,4,2]+p[2,4,4]$//Together)/.$\{x->0\}$

\hspace{.4in} Output: $1$

\vspace{.1in}

$c_{2,4,1,2}$ Input:$(p[2,4,2]x^\wedge2+p[2,4,4]x^\wedge4$//Together)/.$\{x->0\}$

\hspace{.4in} Output: $30$

\vspace{.2in}
Determine the coefficient of powers of $3$:

$c_{2,4,2,1}$ Input:$(p[2,4,3]$//Together)/.$\{x->0\}$

\hspace{.4in} Output: $3$

\vspace{.2in}

Combining  the above results, we have from Theorem 3,  \[_nm_{2,4}=1\cdot 5^{n-1}+30\binom{n-1}{1}5^{n-2}+3\cdot3^{n-1}.\]
\vspace{.3in}

(3) Let $M$ be the matrix

 \parbox{3cm}{$\left(\begin{array}{cccc}
5&2&1&3\\
0&5&4&2\\
0&0&5&1\\
0&0&0&5
\end{array}\right)$}
\vspace{.2in}

The following paragraphs show how to  find a closed form formula for $_nm_{1,4}$ , noting that \mbox{$m_{1,4,1}=5$},  \mbox{$mpy_{1,4}(1)=4$}, and $num(1,4)=1$. 
\vspace{.2in}

\hspace{-1.2in}\mbox{$c_{1,4,1,1}$} Input:
 $(p[1,4,1]+p[1,4,2]+p[1,4,3]+p[1,4,4]$//Together)
/.$\{x->0\}$

\hspace{-.1in} Output: $3$

\vspace{.1in}

\hspace{-1.4in}\mbox{$c_{1,4,1,2}$} Input:$p[1,4,1]x+p[1,4,2]x^\wedge2+p[1,4,3]x^\wedge3+p[1,4,4]x^\wedge4$//Together)/.$\{x->0\}$

\hspace{-.1in} Output: $20$

\vspace{.2in}

\hspace{-1.4in}\mbox{$c_{1,4,1,3}$} Input:$p[1,4,1]x^\wedge2+p[1,4,2]x^\wedge4+p[1,4,3]x^\wedge6+p[1,4,4]x^\wedge8$//Together)/.$\{x->0\}$

\hspace{-.1in} Output: $33$

\vspace{.2in}

\hspace{-1.4in}\mbox{$c_{1,4,1,4}$} Input:$p[1,4,1]x^\wedge3+p[1,4,2]x^\wedge6+p[1,4,3]x^\wedge9+p[1,4,4]x^\wedge12$//Together)/.$\{x->0\}$

\hspace{-.1in} Output: $40$

\vspace{.2in}

Combining  the above results, we have from Theorem 3,
\[_nm_{1,4}=3\cdot5^{n-1}+20 \binom{n-1}{1}5^{n-2}+33 \binom{n-1}{2}5^{n-3}+40\binom{n-1}{3}5^{n-4}.\]

\vspace {.2in}
(Because all of the diagonal elements have the same value, it is interesting to note that the above result can be obtained by thinking of the matrix as a one-way Markov process (notwithstanding the fact that the elements of the matrix are not  between zero and one). Given that the system is in state $i$, $m_{i,j}$ represents the probability of being in state $j$ after the next trial, and $_nm_{1,4}$ represents the probability of being in state $4$ after the  $n^{th}$ trial. $_nm_{1,4}$ is then the sum of the following mutually exclusive probabilities:

$m_{1,4}\cdot 5^{n-1}=3\cdot 5^{n-1}$ represents the probability that $4$ is the only state occurring in the $n$ trials,

($m_{1,1}m_{14}+m_{1,2}m_{2,4} +m_{1,3}m_{3,4})\binom{n-1}{1}5^{n-2}=20\binom{n-1}{1}5^{n-2}$ represents the probability that exactly $2$ states occur in the $n$ trials,

($m_{1,1}m_{1,2}m_{2,4}+m_{1,1}m_{1,3}m_{3,4}+m_{1,2}m_{2,3}m_{3,4})\binom{n-1}{2}5^{n-3}=33\binom{n-1}{2}5^{n-3}$ represents the probability that exactly $3$ states occur in the $n$ trials, and

$m_{1,1}m_{1,2}m_{2,3}m_{3,4} \binom{n-1}{3}5^{n-4}=40 \binom{n-1}{3}5^{n-4}$ represents the probability that exactly $4$ states occur in the $n$ trials.)

\vspace{1in}
\hspace{4in} 4/16/14

\end{document}